\documentclass[12pt,oneside,reqno]{amsart}

\usepackage{amsmath,amssymb,amsthm,textcomp}
\usepackage{amsfonts,graphicx}
\usepackage[mathscr]{eucal}
\usepackage{color}
\usepackage{csquotes}
\usepackage[backend=bibtex,%
firstinits=true,%
doi=false,%
isbn=true,%
url=false,%
maxnames=99]{biblatex}%
\AtEveryBibitem{\clearfield{issn}}
\AtEveryCitekey{\clearfield{issn}}
\numberwithin{equation}{section}
\DeclareNameAlias{sortname}{last-first}
\theoremstyle{definition}

\numberwithin{equation}{section}

\newcommand{\ncom}{\newcommand}

\ncom{\beq}{\begin{equation}}
\ncom{\eeq}{\end{equation}}
\ncom{\bea}{\begin{eqnarray*}}
\ncom{\eea}{\end{eqnarray*}}
\ncom{\beqa}{\begin{eqnarray}}
\ncom{\eeqa}{\end{eqnarray}}
\ncom{\nno}{\nonumber}
\ncom{\non}{\nonumber}
\ncom{\ds}{\displaystyle}
\ncom{\half}{\frac{1}{2}}
\ncom{\mbx}{\makebox{.25cm}}
\ncom{\hs}{\mbox{\hspace{.25cm}}}
\ncom{\rar}{\rightarrow}
\ncom{\Rar}{\Rightarrow}
\ncom{\noin}{\noindent}
\ncom{\bc}{\begin{center}}
\ncom{\ec}{\end{center}}
\ncom{\sz}{\scriptsize}
\ncom{\rf}{\ref}
\ncom{\s}{\sqrt{2}}
\ncom{\sgm}{\sigma}
\ncom{\Sgm}{\Sigma}
\ncom{\psgm}{\sigma^{\prime}}
\ncom{\dt}{\delta}
\ncom{\Dt}{\Delta}
\ncom{\lmd}{\lambda}
\ncom{\Lmd}{\Lambda}
\ncom{\Th}{\Theta}
\ncom{\e}{\eta}
\ncom{\eps}{\epsilon}
\ncom{\pcc}{\stackrel{P}{>}}
\ncom{\lp}{\stackrel{L_{p}}{>}}
\ncom{\dist}{{\rm\,dist}}
\ncom{\sspan}{{\rm\,span}}
\ncom{\re}{{\rm Re\,}}
\ncom{\im}{{\rm Im\,}}
\ncom{\sgn}{{\rm sgn\,}}
\ncom{\ba}{\begin{array}}
\ncom{\ea}{\end{array}}
\ncom{\hone}{\mbox{\hspace{1em}}}
\ncom{\htwo}{\mbox{\hspace{2em}}}
\ncom{\hthree}{\mbox{\hspace{3em}}}
\ncom{\hfour}{\mbox{\hspace{4em}}}
\ncom{\vone}{\vskip 2ex}
\ncom{\vtwo}{\vskip 4ex}
\ncom{\vonee}{\vskip 1.5ex}
\ncom{\vthree}{\vskip 6ex}
\ncom{\vfour}{\vspace*{8ex}}
\ncom{\norm}{\|\;\;\|}
\ncom{\integ}[4]{\int_{#1}^{#2}\,{#3}\,d{#4}}
\ncom{\vspan}[1]{{{\rm\,span}\{ #1 \}}}
\ncom{\dm}[1]{ {\displaystyle{#1} } }
\ncom{\ri}[1]{{#1} \index{#1}}

\newtheorem{theorem}{\bf Theorem}[section]
\newtheorem{remark}{\bf Remark}[section]

\newtheorem{lemma}{Lemma}[section]
\newtheorem{corollary}{Corollary}[section]

\newtheoremstyle
    {remarkstyle}
    {}
    {11pt}
    {}
    {}
    {\bfseries}
    {:}
    {     }
    {\thmname{#1} \thmnumber{#2} }

\theoremstyle{remarkstyle}

\def\eps{\varepsilon}

\begin{document}
\title{Some Fractional Calculus results\\ associated with the $I$-Function}
\author[Kuldeep Kumar Kataria]{K. K. Kataria}
\address{Kuldeep Kumar Kataria, Department of Mathematics,
 Indian Institute of Technology Bombay, Powai, Mumbai 400076, INDIA.}
 \email{kulkat@math.iitb.ac.in}
\author{P. Vellaisamy}
\address{P. Vellaisamy, Department of Mathematics,
 Indian Institute of Technology Bombay, Powai, Mumbai 400076, INDIA.}
 \email{pv@math.iitb.ac.in}
\thanks{The research of K. K. Kataria was supported by UGC, Govt. of India.}
\begin{abstract}
The effect of Marichev-Saigo-Maeda (MSM) fractional operators involving third Appell function on the $I$ function is studied. It is shown that the order of the $I$-function increases on application of these operators to the power multiple of the $I$-function. The Caputo-type MSM fractional derivatives are introduced and studied for the $I$-function. As special cases, the corresponding assertions for Saigo and Erd\'elyi-Kober fractional operators are also presented. The results obtained in this paper generalize several known results obtained recently in the literature. 
\end{abstract}

\maketitle
\section{Introduction}
The Fox's $H$-function \cite{Fox395}  has found enormous application in the fields of statistical distributions, fractional calculus, statistical mechanics and some areas of engineering. Inayat-Hussain \cite{Hussain4119} generalized Fox's $H$-function namely to $\overline{H}$-function. Generalized Riemann zeta functions, poly-logarithms function of complex order and the exact partition functions of the  Gaussian model in statistical mechanics, among others, are the special cases of $\overline{H}$-function, which are not particular cases of $H$-function. Rathie \cite{Rathie297} introduced $I$-function, which includes $\overline{H}$-function as a special case. The $I$-function is represented by the following Mellin-Barnes type contour integral
\begin{equation}\label{1.2}
I^{m,n}_{p,q}\Bigg[z\left|
\begin{matrix}
    (a_i,A_i,\alpha_i)_{1,p}\\ 
    (b_j,B_j,\beta_j)_{1,q}
  \end{matrix}
\right.\Bigg]=\frac{1}{2\pi i}\int_{C}\chi(s)z^{-s}\,ds,
\end{equation}
where
\begin{equation}\label{n1.1}
\chi(s)=\frac{\prod_{j=1}^{m}\Gamma^{\beta_j}\left(b_j+B_js\right)\prod_{i=1}^{n}\Gamma^{\alpha_i}\left(1-a_i-A_is\right)}{\prod_{j=m+1}^{q}\Gamma^{\beta_j}\left(1-b_j-B_js\right)\prod_{i=n+1}^{p}\Gamma^{\alpha_i}\left(a_i+A_is\right)}.
\end{equation}
In the above definition, $z\neq0,m, n, p, q$ are integers satisfying $0 \leq m \leq q$ and $0 \leq n \leq p$ with $\alpha_i,A_i>0$ for $i=1,2,\ldots,p$ and $\beta_j,B_j>0$ for $j=1,2,\ldots,q$. Also, $a_i$'s and $b_j$'s are complex numbers such that no singularity of $\Gamma^{\beta_j}\left(b_j+B_js\right)$ coincides with any singularity of $\Gamma^{\alpha_i}\left(1-a_i-A_is\right)$. An empty product is to be interpreted as unity. The path of integration $C$, in the complex $s$-plane runs from $c-i\infty$ to $c+i\infty$ for some real number $c$ such that the singularity of $\Gamma^{\beta_j}\left(b_j+B_js\right)$ lie entirely to the right of the path and the singularity of $\Gamma^{\alpha_i}\left(1-a_i-A_is\right)$ lie entirely to the left of the path. For other possible contours $C$ and convergence conditions, various series representations, elementary properties of $I$-function one can refer \cite{Rathie297}. The $\overline{H}$-function follows as a particular case when $\alpha_i=1$ for $i=n+1,\ldots,p$ and $\beta_j=1$ for $j=1,2,\ldots,m$ in (\ref{1.2}). Recently, the extension of $I$-function to several complex variables (see \cite{Kumari285}, \cite{Prathima}) and its applications in wireless communication were studied by several authors (see \cite{Ansari394}, \cite{Ansari1}, \cite{Minghua5578}). 

In this paper, we derive some fractional calculus results associated with the $I$ function. For $\alpha,\alpha',\beta,\beta',\gamma\in\mathbb{C}$ and $x>0$ with $\operatorname{Re}(\gamma)>0$, the left- and right-hand sided Marichev-Saigo-Maeda (MSM) fractional integral operators \cite{Marichev1974128} associated with third Appell function are defined by
\begin{equation}\label{1.13}
\left(\mathcal{I}_{0+}^{\alpha,\alpha',\beta,\beta',\gamma}f\right)(x)=\frac{x^{-\alpha}}{\Gamma(\gamma)}\int_0^x(x-t)^{\gamma-1}t^{-\alpha'}F_3\left(\alpha,\alpha',\beta,\beta';\gamma;1-\frac{t}{x},1-\frac{x}{t}\right)f(t)\,dt
\end{equation}
and
\begin{equation}\label{1.14}
\left(\mathcal{I}_{-}^{\alpha,\alpha',\beta,\beta',\gamma}f\right)(x)=\frac{x^{-\alpha'}}{\Gamma(\gamma)}\int_x^{\infty}(t-x)^{\gamma-1}t^{-\alpha}F_3\left(\alpha,\alpha',\beta,\beta';\gamma;1-\frac{x}{t},1-\frac{t}{x}\right)f(t)\,dt,
\end{equation}
respectively. The corresponding fractional differential operators \cite{Saigo1998} have their respective forms as
\begin{equation}\label{1.15}
\left(\mathcal{D}_{0+}^{\alpha,\alpha',\beta,\beta',\gamma}f\right)(x)=\left(\frac{d}{dx}\right)^{[\operatorname{Re}(\gamma)]+1}\left(\mathcal{I}_{0+}^{-\alpha',-\alpha,-\beta'+[\operatorname{Re}(\gamma)]+1,-\beta,-\gamma+[\operatorname{Re}(\gamma)]+1}f\right)(x)
\end{equation}
and
\begin{equation}\label{1.16}
\left(\mathcal{D}_{-}^{\alpha,\alpha',\beta,\beta',\gamma}f\right)(x)=\left(-\frac{d}{dx}\right)^{[\operatorname{Re}(\gamma)]+1}\left(\mathcal{I}_{-}^{-\alpha',-\alpha,-\beta',-\beta+[\operatorname{Re}(\gamma)]+1,-\gamma+[\operatorname{Re}(\gamma)]+1}f\right)(x),
\end{equation}
where $[\operatorname{Re}(\alpha)]$ denotes the integer part of $\operatorname{Re}(\alpha)$. The third Appell function $F_3$ \cite{Prudnikov1990} (also known as Horn function), is defined by
\begin{equation*}\label{1.17}
F_3(\alpha,\alpha',\beta,\beta';\gamma;x;y)=\sum_{m,n=0}^{\infty}\frac{(\alpha)_{m}(\alpha')_{n}(\beta)_{m}(\beta')_{n}}{(\gamma)_{m+n}}\frac{x^my^n}{m!n!},
\end{equation*}
such that $\max\{|x|,|y|\}<1$. Here $(z)_{n}$ is the Pochhammer symbol, defined for $z\in\mathbb{C}$ by
\begin{equation*}
(z)_{n}=\left\{
	\begin{array}{ll}
	    1 & \mbox{if } n=0,\\
		z(z+1)(z+2)\ldots(z+(n-1))  & \mbox{if } n\in\mathbb{N}.
	\end{array}
\right.
\end{equation*}
Saigo \cite{Saigo1978135} introduced the fractional integral and differential operators involving Gauss hypergeometric function ${}_{2}F_1$ as the kernel.  For $\alpha,\beta,\gamma\in\mathbb{C}$ and $x>0$ with $\operatorname{Re}(\alpha)>0$, the left- and right-hand sided Saigo fractional integral operators are defined by
\begin{equation}\label{1.8}
\left(\mathcal{I}_{0+}^{\alpha,\beta,\gamma}f\right)(x)=\frac{x^{-\alpha-\beta}}{\Gamma(\alpha)}\int_0^x(x-t)^{\alpha-1}{}_{2}F_1\left(\alpha+\beta,-\gamma;\alpha;1-\frac{t}{x}\right)f(t)\,dt
\end{equation}
and
\begin{equation}\label{1.9}
\left(\mathcal{I}_{-}^{\alpha,\beta,\gamma}f\right)(x)=\frac{1}{\Gamma(\alpha)}\int_x^{\infty}(t-x)^{\alpha-1}t^{-\alpha-\beta}{}_{2}F_1\left(\alpha+\beta,-\gamma;\alpha;1-\frac{x}{t}\right)f(t)\,dt,
\end{equation}
respectively. 
The corresponding fractional differential operators are
\begin{equation}\label{1.10}
\left(\mathcal{D}_{0+}^{\alpha,\beta,\gamma}f\right)(x)=\left(\frac{d}{dx}\right)^{[\operatorname{Re}(\alpha)]+1}\left(\mathcal{I}_{0+}^{-\alpha+[\operatorname{Re}(\alpha)]+1,-\beta-[\operatorname{Re}(\alpha)]-1,\alpha+\gamma-[\operatorname{Re}(\alpha)]-1}f\right)(x)
\end{equation}
and
\begin{equation}\label{1.11}
\left(\mathcal{D}_{-}^{\alpha,\beta,\gamma}f\right)(x)=\left(-\frac{d}{dx}\right)^{[\operatorname{Re}(\alpha)]+1}\left(\mathcal{I}_{-}^{-\alpha+[\operatorname{Re}(\alpha)]+1,-\beta-[\operatorname{Re}(\alpha)]-1,\alpha+\gamma}f\right)(x).
\end{equation}
For $\beta=-\alpha$ and $\beta=0$ in (\ref{1.8})-(\ref{1.11}), we get the corresponding Riemann-Liouville and Erd\'elyi-Kober fractional operators  respectively (for definition see \cite{Kilbas2006}). The Gauss hypergeometric function is related to third Appell function as
\begin{equation*}\label{1.18}
F_3(\alpha,\gamma-\alpha,\beta,\gamma-\beta;\gamma;x;y)={}_{2}F_1\left(\alpha,\beta,\gamma;x+y-xy\right).
\end{equation*}
The MSM fractional operators (\ref{1.13})-(\ref{1.16}) are connected to Saigo operators (\ref{1.8})-(\ref{1.11}) by
\begin{equation}\label{1.19}
\left(\mathcal{I}_{0+}^{\alpha,0,\beta,\beta',\gamma}f\right)(x)=\left(\mathcal{I}_{0+}^{\gamma,\alpha-\gamma,-\beta}f\right)(x),\ \ \ \ 
\left(\mathcal{I}_{-}^{\alpha,0,\beta,\beta',\gamma}f\right)(x)=\left(\mathcal{I}_{-}^{\gamma,\alpha-\gamma,-\beta}f\right)(x)
\end{equation}
and
\begin{equation}\label{1.20}
\left(\mathcal{D}_{0+}^{0,\alpha',\beta,\beta',\gamma}f\right)(x)=\left(\mathcal{D}_{0+}^{\gamma,\alpha'-\gamma,\beta'-\gamma}f\right)(x),\ \ \ \ 
\left(\mathcal{D}_{-}^{0,\alpha',\beta,\beta',\gamma}f\right)(x)=\left(\mathcal{D}_{-}^{\gamma,\alpha'-\gamma,\beta'-\gamma}f\right)(x).
\end{equation}
In Section 2, some preliminary results are stated, which will be used in proofs of subsequent theorems.

\section{Preliminaries}
\setcounter{equation}{0}
\noindent The following notations are used throughout the paper:
\begin{eqnarray*}\label{2.1}
\mu&=&\sum_{j=1}^q\beta_jB_j-\sum_{i=1}^p\alpha_iA_i,\\
\Omega&=&\sum_{i=1}^p\left(\frac{1}{2}-\operatorname{Re}(a_i)\right)\alpha_i-\sum_{j=1}^q\left(\frac{1}{2}-\operatorname{Re}(b_j)\right)\beta_j,\\
\Delta&=&\sum_{j=1}^m\beta_jB_j-\sum_{j=m+1}^q\beta_jB_j+\sum_{i=1}^n\alpha_iA_i-\sum_{i=n+1}^p\alpha_iA_i,
\end{eqnarray*}
where $m,n$, $a_i,A_i,\alpha_i$ and $b_j,B_j,\beta_j$ appear in the definition of $I$-function (see \ref{1.2}). The $I$-function is analytic if $\mu\geq0$ and the integral in (\ref{1.2}) converges absolutely if $|\arg(z)|<\Delta\frac{\pi}{2}$, where $\Delta>0$. Also, if $|\arg(z)|=\Delta\frac{\pi}{2}$ with $\Delta\geq0$, then it converges absolutely under the following conditions (see \cite{Rathie297}): (i) $\mu=0$ and $\Omega<-1$. (ii) $|\mu|\neq0$ with $s=\sigma+it$, where $\sigma,t\in\mathbb{R}$ and are such that for $|t|\rightarrow\infty$ we have $\Omega+\sigma\mu<-1$.\\
\noindent The following are well known results for MSM integral operators of power functions (see \cite{Saigo1998}).
\begin{lemma}\label{l2.1}
Let $\alpha,\alpha',\beta,\beta',\gamma,\rho\in\mathbb{C}$ such that $\operatorname{Re}(\gamma)>0$.\\
\noindent (a) If $\operatorname{Re}(\rho)>$ $\max\{0,\operatorname{Re}(\alpha'-\beta'),\operatorname{Re}(\alpha+\alpha'+\beta-\gamma)\}$, then
\begin{equation}\label{2.4}
\left(I_{0+}^{\alpha,\alpha',\beta,\beta',\gamma}t^{\rho-1}\right)(x)=\frac{\Gamma(\rho)\Gamma(-\alpha'+\beta'+\rho)\Gamma(-\alpha-\alpha'-\beta+\gamma+\rho)}{\Gamma(\beta'+\rho)\Gamma(-\alpha-\alpha'+\gamma+\rho)\Gamma(-\alpha'-\beta+\gamma+\rho)}x^{-\alpha-\alpha'+\gamma+\rho-1}.
\end{equation}
\noindent (b) If $\operatorname{Re}(\rho)>$ $\max\{\operatorname{Re}(\beta),\operatorname{Re}(-\alpha-\alpha'+\gamma),\operatorname{Re}(-\alpha-\beta'+\gamma)\}$, then
\begin{equation}\label{2.5}
\left(I_{-}^{\alpha,\alpha',\beta,\beta',\gamma}t^{-\rho}\right)(x)=\frac{\Gamma(-\beta+\rho)\Gamma(\alpha+\alpha'-\gamma+\rho)\Gamma(\alpha+\beta'-\gamma+\rho)}{\Gamma(\rho)\Gamma(\alpha-\beta+\rho)\Gamma(\alpha+\alpha'+\beta'-\gamma+\rho)}x^{-\alpha-\alpha'+\gamma-\rho}.
\end{equation}
\end{lemma}
\noindent In subsequent theorems, the conditions for the absolute convergence of the integral involved in (\ref{1.2}) is assumed. Also, the contour of integration $C$ is assumed to be the imaginary axis \textit{i.e.} $\operatorname{Re}(s)=0$ .

\section{The MSM fractional integration of $I$-function}
\setcounter{equation}{0}
\noindent First, we present the left-hand sided MSM fractional integration of the $I$- function.
\begin{theorem}\label{t3.1}
Let $\alpha,\alpha',\beta,\beta',\gamma,\rho,a\in\mathbb{C}$ be such that $\operatorname{Re}(\gamma),\mu>0$ and $\operatorname{Re}\left(\rho\right)>\max\{0,$ $\operatorname{Re}(\alpha'-\beta'),\operatorname{Re}(\alpha+\alpha'+\beta-\gamma)\}$. Then for $x>0$,
\begin{eqnarray}\label{3.1}
&&\left(\mathcal{I}_{0+}^{\alpha,\alpha',\beta,\beta',\gamma}\left(t^{\rho-1}I^{m,n}_{p,q}\Bigg[at^{\mu}\left|
\begin{matrix}
    (a_i,A_i,\alpha_i)_{1,p}\\ 
    (b_j,B_j,\beta_j)_{1,q}
  \end{matrix}
\right.\Bigg]\right)\right)(x)\nonumber\\
&&\ \ \ \ \ \ \ \ \ \ =x^{-\alpha-\alpha'+\gamma+\rho-1}I^{m,n+3}_{p+3,q+3}\Bigg[ax^{\mu}\left|
\begin{matrix}
    (1-\rho,\mu,1)&(1+\alpha'-\beta'-\rho,\mu,1)\\ 
    (b_j,B_j,\beta_j)_{1,q}&(1-\beta'-\rho,\mu,1)
  \end{matrix}\right.\nonumber\\
&&\ \ \ \ \ \ \ \ \ \ \ \ \ \ \ \ \ \ \ 
\begin{matrix}
   (1+\alpha+\alpha'+\beta-\gamma-\rho,\mu,1)&(a_i,A_i,\alpha_i)_{1,p}\\ 
   (1+\alpha+\alpha'-\gamma-\rho,\mu,1)&(1+\alpha'+\beta-\gamma-\rho,\mu,1)
  \end{matrix}
\Bigg].
\end{eqnarray}
\end{theorem}
\begin{proof}
The lhs of (\ref{3.1}) is given by
\begin{equation}\label{3.2}
\left(\mathcal{I}_{0+}^{\alpha,\alpha',\beta,\beta',\gamma}\left(t^{\rho-1}\frac{1}{2\pi i}\int_{C}\chi(s)(at^{\mu})^{-s}\,ds,\right)\right)(x),
\end{equation}
where $\chi(s)$ is given by (\ref{n1.1}). Interchanging the order of integration and using (\ref{2.4}), (\ref{3.2}) is equal to 
\begin{equation*}
\frac{1}{2\pi i}\int_{C}\chi(s)a^{-s}\left(\mathcal{I}_{0+}^{\alpha,\alpha',\beta,\beta',\gamma}t^{\rho-\mu s-1}\right)(x)\,ds=x^{-\alpha-\alpha'+\gamma+\rho-1}\frac{1}{2\pi i}\int_{C}\chi(s)\chi_1(s)(ax^{\mu})^{-s}\,ds,
\end{equation*}
where
\begin{equation*}
\chi_1(s)=\frac{\Gamma\left(\rho-\mu s\right)\Gamma\left(-\alpha'+\beta'+\rho-\mu s\right)\Gamma\left(-\alpha-\alpha'-\beta+\gamma+\rho-\mu s\right)}{\Gamma\left(\beta'+\rho-\mu s\right)\Gamma\left(-\alpha-\alpha'+\gamma+\rho-\mu s\right)\Gamma\left(-\alpha'-\beta+\gamma+\rho-\mu s\right)}.
\end{equation*}
The results now follows from (\ref{1.2}).
\end{proof}
\noindent In view of (\ref{1.19}), we have the following result for Saigo operators.
\begin{corollary}\label{c3.1}
Let $\alpha,\beta,\gamma,\rho,a\in\mathbb{C}$ be such that $\operatorname{Re}(\alpha),\mu>0$ and $\operatorname{Re}\left(\rho\right)>\max\{0,$ $\operatorname{Re}(\beta-\gamma)\}$. Then the left-hand sided generalized fractional integration $\mathcal{I}_{0+}^{\alpha,\beta,\gamma}$ of the $I$-function is given for $x>0$ by
\begin{eqnarray*}\label{3.3}
&&\left(\mathcal{I}_{0+}^{\alpha,\beta,\gamma}\left(t^{\rho-1}I^{m,n}_{p,q}\Bigg[at^{\mu}\left|
\begin{matrix}
    (a_i,A_i,\alpha_i)_{1,p}\\ 
    (b_j,B_j,\beta_j)_{1,q}
  \end{matrix}
\right.\Bigg]\right)\right)(x)\nonumber\\
&&\ \ =x^{-\beta+\rho-1}I^{m,n+2}_{p+2,q+2}\Bigg[ax^{\mu}\left|
\begin{matrix}
    (1-\rho,\mu,1)&(1+\beta-\gamma-\rho,\mu,1)&(a_i,A_i,\alpha_i)_{1,p}\\ 
    (b_j,B_j,\beta_j)_{1,q}&(1+\beta-\rho,\mu,1)&(1-\alpha-\gamma-\rho,\mu,1)
  \end{matrix}\right.\Bigg].
\end{eqnarray*}
\end{corollary}
\noindent The above corollary leads to Erd\'elyi-Kober fractional integral as follows.
\begin{corollary}\label{c3.3}
Let $\alpha,\gamma,\rho,a\in\mathbb{C}$ be such that $\operatorname{Re}(\alpha),\mu>0$ and $\operatorname{Re}\left(\rho\right)>\max\{0,$ $\operatorname{Re}(-\gamma)\}$. Then the left-hand sided Erd\'elyi-Kober fractional integration $\mathcal{I}_{\gamma,\alpha}^{+}$ $(=\mathcal{I}_{0+}^{\alpha,0,\gamma})$ of the $I$-function is given for $x>0$ by
\begin{eqnarray*}\label{3.5}
&&\left(\mathcal{I}_{\gamma,\alpha}^{+}\left(t^{\rho-1}I^{m,n}_{p,q}\Bigg[at^{\mu}\left|
\begin{matrix}
    (a_i,A_i,\alpha_i)_{1,p}\\ 
    (b_j,B_j,\beta_j)_{1,q}
  \end{matrix}
\right.\Bigg]\right)\right)(x)\nonumber\\
&&\ \ \ \ \ \ \ \ \ \ \ \ \ \ \ \ \ \ \ \ \ \ \ \ \ \ \ \ \ =x^{\rho-1}I^{m,n+1}_{p+1,q+1}\Bigg[ax^{\mu}\left|
\begin{matrix}
    (1-\gamma-\rho,\mu,1)&(a_i,A_i,\alpha_i)_{1,p}\\ 
    (b_j,B_j,\beta_j)_{1,q}&(1-\alpha-\gamma-\rho,\mu,1)
  \end{matrix}\right.\Bigg].
\end{eqnarray*}
\end{corollary}
\noindent The following result corresponds to the right-hand sided MSM fractional integration of the $I$-function.
\begin{theorem}\label{t3.2}
Let $\alpha,\alpha',\beta,\beta',\gamma,\rho,a\in\mathbb{C}$ be such that $\operatorname{Re}(\gamma),\mu>0$ and $\operatorname{Re}\left(\rho\right)>\max\{\operatorname{Re}(\beta),$ $\operatorname{Re}(-\alpha-\alpha'+\gamma),\operatorname{Re}(-\alpha-\beta'+\gamma)\}$. Then
\begin{eqnarray}\label{3.6}
&&\left(\mathcal{I}_{-}^{\alpha,\alpha',\beta,\beta',\gamma}\left(t^{-\rho}I^{m,n}_{p,q}\Bigg[at^{-\mu}\left|
\begin{matrix}
    (a_i,A_i,\alpha_i)_{1,p}\\ 
    (b_j,B_j,\beta_j)_{1,q}
  \end{matrix}
\right.\Bigg]\right)\right)(x)\nonumber\\
&&\ \ \ \ \ \ \ \ \ \ =x^{-\alpha-\alpha'+\gamma-\rho}I^{m,n+3}_{p+3,q+3}\Bigg[ax^{-\mu}\left|
\begin{matrix}
    (1+\beta-\rho,\mu,1)&(1-\alpha-\alpha'+\gamma-\rho,\mu,1)\\ 
    (b_j,B_j,\beta_j)_{1,q}&(1-\rho,\mu,1)
  \end{matrix}\right.\nonumber\\
&&\ \ \ \ \ \ \ \ \ \ \ \ \ \ \ \ \ \ \ \ \ \ 
\begin{matrix}
   (1-\alpha-\beta'+\gamma-\rho,\mu,1)&(a_i,A_i,\alpha_i)_{1,p}\\ 
   (1-\alpha+\beta-\rho,\mu,1)&(1-\alpha-\alpha'-\beta'+\gamma-\rho,\mu,1)
  \end{matrix}
\Bigg],
\end{eqnarray}
for $x>0$.
\end{theorem}
\begin{proof}
Using (\ref{1.2}), the lhs of (\ref{3.6}) is equal to
\begin{equation}\label{3.7}
\left(\mathcal{I}_{-}^{\alpha,\alpha',\beta,\beta',\gamma}\left(t^{-\rho}\frac{1}{2\pi i}\int_{C}\chi(s)(at^{-\mu})^{-s}\,ds\right)\right)(x),
\end{equation}
where $\chi(s)$ is given by (\ref{n1.1}). Interchanging the order of integration and using (\ref{2.5}), (\ref{3.7}) reduces to 
\begin{equation*}
\frac{1}{2\pi i}\int_{C}\chi(s)a^{-s}\left(\mathcal{I}_{-}^{\alpha,\alpha',\beta,\beta',\gamma}t^{-(\rho-\mu s)}\right)(x)\,ds=x^{-\alpha-\alpha'+\gamma-\rho}\frac{1}{2\pi i}\int_{C}\chi(s)\chi_2(s)(ax^{-\mu})^{-s}\,ds,
\end{equation*}
where
\begin{equation*}
\chi_2(s)=\frac{\Gamma\left(-\beta+\rho-\mu s\right)\Gamma\left(\alpha+\alpha'-\gamma+\rho-\mu s\right)\Gamma\left(\alpha+\beta'-\gamma+\rho-\mu s\right)}{\Gamma\left(\rho-\mu s\right)\Gamma\left(\alpha-\beta+\rho-\mu s\right)\Gamma\left(\alpha+\alpha'+\beta'-\gamma+\rho-\mu s\right)}.
\end{equation*}
The result follows from (\ref{1.2}).
\end{proof}
\noindent The Saigo and Erd\'elyi-Kober fractional integration of the $I$-function follow as corollaries.
\begin{corollary}\label{c3.4}
Let $\alpha,\beta,\gamma,\rho,a\in\mathbb{C}$ be such that $\operatorname{Re}(\alpha),\mu>0$ and $\operatorname{Re}\left(\rho\right)>\max\{\operatorname{Re}(-\beta),$ $\operatorname{Re}(-\gamma)\}$. Then the right-hand sided generalized fractional integration $\mathcal{I}_{-}^{\alpha,\beta,\gamma}$ of the $I$-function, for $x>0$, is given by
\begin{eqnarray*}\label{3.8}
&&\left(\mathcal{I}_{-}^{\alpha,\beta,\gamma}\left(t^{-\rho}I^{m,n}_{p,q}\Bigg[at^{-\mu}\left|
\begin{matrix}
    (a_i,A_i,\alpha_i)_{1,p}\\ 
    (b_j,B_j,\beta_j)_{1,q}
  \end{matrix}
\right.\Bigg]\right)\right)(x)\nonumber\\
&&=x^{-\beta-\rho}I^{m,n+2}_{p+2,q+2}\Bigg[ax^{-\mu}\left|
\begin{matrix}
    (1-\gamma-\rho,\mu,1)&(1-\beta-\rho,\mu,1)&(a_i,A_i,\alpha_i)_{1,p}\\ 
    (b_j,B_j,\beta_j)_{1,q}&(1-\rho,\mu,1)&(1-\alpha-\beta-\gamma-\rho,\mu,1)
  \end{matrix}\right.\Bigg].
\end{eqnarray*}
\end{corollary}
\begin{corollary}\label{c3.6}
Let $\alpha,\gamma,\rho,a\in\mathbb{C}$ be such that $\operatorname{Re}(\alpha),\mu>0$ and $\operatorname{Re}\left(\rho\right)>\max\{0,$ $\operatorname{Re}(-\gamma)\}$. Then the right-hand sided Erd\'elyi-Kober fractional integration $\mathcal{K}_{\gamma,\alpha}^{-}$ $(=\mathcal{I}_{-}^{\alpha,0,\gamma})$ of $I$-function is given for $x>0$ by
\begin{eqnarray*}\label{3.10}
&&\left(\mathcal{K}_{\gamma,\alpha}^{-}\left(t^{-\rho}I^{m,n}_{p,q}\Bigg[at^{-\mu}\left|
\begin{matrix}
    (a_i,A_i,\alpha_i)_{1,p}\\ 
    (b_j,B_j,\beta_j)_{1,q}
  \end{matrix}
\right.\Bigg]\right)\right)(x)\nonumber\\
&&\ \ \ \ \ \ \ \ \ \ \ \ \ \ \ \ \ \ \ \ \ \ \ \ \ \ \ \ \ =x^{-\rho}I^{m,n+1}_{p+1,q+1}\Bigg[ax^{-\mu}\left|
\begin{matrix}
    (1-\gamma-\rho,\mu,1)&(a_i,A_i,\alpha_i)_{1,p}\\ 
    (b_j,B_j,\beta_j)_{1,q}&(1-\alpha-\gamma-\rho,\mu,1)
  \end{matrix}\right.\Bigg].
\end{eqnarray*}
\end{corollary}

\section{The MSM fractional differentiation of $I$-function}
\setcounter{equation}{0}
\noindent In this section, we derive the MSM fractional derivative of the $I$-function. We first require the following result.
\begin{lemma}\label{ll4.1}
Let $\alpha,\alpha',\beta,\beta',\gamma,\rho\in\mathbb{C}$.\\
\noindent (a) If $\operatorname{Re}(\rho)>$ $\max\{0,,\operatorname{Re}(-\alpha+\beta),\operatorname{Re}(-\alpha-\alpha'-\beta'+\gamma)\}$, then
\begin{equation}\label{l4.1}
\left(\mathcal{D}_{0+}^{\alpha,\alpha',\beta,\beta',\gamma}t^{\rho-1}\right)(x)=\frac{\Gamma(\rho)\Gamma(\alpha-\beta+\rho)\Gamma(\alpha+\alpha'+\beta'-\gamma+\rho)}{\Gamma(-\beta+\rho)\Gamma(\alpha+\alpha'-\gamma+\rho)\Gamma(\alpha+\beta'-\gamma+\rho)}x^{\alpha+\alpha'-\gamma+\rho-1}.
\end{equation}
\noindent (b) If $\operatorname{Re}(\rho)>$ $\max\{\operatorname{Re}(-\beta'),\operatorname{Re}(\alpha'+\beta-\gamma),\operatorname{Re}(\alpha+\alpha'-\gamma)+[\operatorname{Re}(\gamma)]+1\}$, then
\begin{equation}\label{l4.2}
\left(\mathcal{D}_{-}^{\alpha,\alpha',\beta,\beta',\gamma}t^{-\rho}\right)(x)=\frac{\Gamma\left(\beta'+\rho\right)\Gamma\left(-\alpha-\alpha'+\gamma+\rho\right)\Gamma\left(-\alpha'-\beta+\gamma+\rho\right)}{\Gamma\left(\rho\right)\Gamma\left(-\alpha'+\beta'+\rho\right)\Gamma\left(-\alpha-\alpha'-\beta+\gamma+\rho\right)}x^{\alpha+\alpha'-\gamma-\rho}.
\end{equation} 
\end{lemma}
\begin{proof}
(a) Let $m=[\operatorname{Re}(\gamma)]+1$. Using (\ref{1.15}) and (\ref{2.4}), the lhs of (\ref{l4.1}) is equal to
\begin{eqnarray*}
&&\left(\frac{d}{\,dx}\right)^m\left(\mathcal{I}_{0+}^{-\alpha',-\alpha,-\beta'+m,-\beta,-\gamma+m}t^{\rho-1}\right)(x)\\
&&\ \ \ \ \ \ =\frac{d^m}{\,dx^m}\frac{\Gamma\left(\rho\right)\Gamma\left(\alpha-\beta+\rho\right)\Gamma\left(\alpha+\alpha'+\beta'-\gamma+\rho\right)}{\Gamma\left(-\beta+\rho\right)\Gamma\left(\alpha+\alpha'-\gamma+\rho+m\right)\Gamma\left(\alpha+\beta'-\gamma+\rho\right)}x^{\alpha+\alpha'-\gamma+\rho+m-1}\\
&&\ \ \ \ \ \ =\frac{\Gamma\left(\rho\right)\Gamma\left(\alpha-\beta+\rho\right)\Gamma\left(\alpha+\alpha'+\beta'-\gamma+\rho\right)}{\Gamma\left(-\beta+\rho\right)\Gamma\left(\alpha+\alpha'-\gamma+\rho+m\right)\Gamma\left(\alpha+\beta'-\gamma+\rho\right)}\frac{d^m}{\,dx^m}x^{\alpha+\alpha'-\gamma+\rho+m-1},
\end{eqnarray*}
which on differentiation yields (\ref{l4.1}).\\
\noindent (b) Using (\ref{1.16}) and (\ref{2.5}), the lhs of (\ref{l4.2}) reduces to
\begin{eqnarray}\label{l4.7}
&&\left(-\frac{d}{\,dx}\right)^m\left(\mathcal{I}_{-}^{-\alpha',-\alpha,-\beta',-\beta+m,-\gamma+m}t^{-\rho}\right)(x)\nonumber\\
&&\ \ \ \ =\frac{\Gamma\left(\beta'+\rho\right)\Gamma\left(-\alpha-\alpha'+\gamma+\rho-m\right)\Gamma\left(-\alpha'-\beta+\gamma+\rho\right)}{(-1)^{-m}\Gamma\left(\rho\right)\Gamma\left(-\alpha'+\beta'+\rho\right)\Gamma\left(-\alpha-\alpha'-\beta+\gamma+\rho\right)}\frac{d^m}{\,dx^m}x^{\alpha+\alpha'-\gamma-\rho+m}.\nonumber\\
\end{eqnarray}
Also note that
\begin{equation}\label{l4.8}
\frac{d^m}{\,dx^m}x^{\alpha+\alpha'-\gamma-\rho+m}=\frac{\Gamma\left(\alpha+\alpha'-\gamma-\rho+m+1\right)}{\Gamma\left(\alpha+\alpha'-\gamma-\rho+1\right)}x^{\alpha+\alpha'-\gamma-\rho}.
\end{equation}
By the reflection formula for the gamma function  (see \cite{Erdelyi1953}), we obtain
\begin{eqnarray}\label{l4.9}
&&\Gamma\left(-\alpha-\alpha'+\gamma+\rho-m\right)\Gamma\left(1-\left(-\alpha-\alpha'+\gamma+\rho-m\right)\right)\nonumber\\
&&\ \ \ \ \ \ \ =\frac{\pi}{\sin{\left(-\alpha-\alpha'+\gamma+\rho-m\right)\pi}}=\frac{\pi}{(-1)^m\sin{\left(-\alpha-\alpha'+\gamma+\rho\right)\pi}}.
\end{eqnarray}
Similarly,
\begin{equation}\label{l4.10}
\Gamma\left(-\alpha-\alpha'+\gamma+\rho\right)\Gamma\left(1-\left(-\alpha-\alpha'+\gamma+\rho\right)\right)=\frac{\pi}{\sin{\left(-\alpha-\alpha'+\gamma+\rho\right)\pi}}.
\end{equation}
On substituting (\ref{l4.8})-(\ref{l4.10}) in (\ref{l4.7}), the result follows.
\end{proof}
\begin{remark}\label{r4.1}
The following connections between Lemma \ref{l2.1} and Lemma \ref{ll4.1} are clear in view of (\ref{1.15}) and (\ref{1.16}).\\
\noindent (a) If in the hypothesis of Lemma \ref{l2.1}(a), we make the changes $$\alpha\rightarrow-\alpha',\ \alpha'\rightarrow-\alpha,\ \beta\rightarrow-\beta'+[\operatorname{Re}(\gamma)]+1,\ \beta'\rightarrow-\beta,\ \gamma\rightarrow-\gamma+[\operatorname{Re}(\gamma)]+1$$ and in the rhs of (\ref{l4.1}), $\alpha\rightarrow-\alpha'$, $\alpha'\rightarrow-\alpha$, $\beta\rightarrow-\beta'$, $\beta'\rightarrow-\beta$, $\gamma\rightarrow-\gamma$ respectively, then Lemma \ref{ll4.1}(a) follows.\\
\noindent (b) If in the hypothesis of Lemma \ref{l2.1}(b), we make the changes $$\alpha\rightarrow-\alpha',\ \alpha'\rightarrow-\alpha,\ \beta\rightarrow-\beta',\ \beta'\rightarrow-\beta+[\operatorname{Re}(\gamma)]+1,\ \gamma\rightarrow-\gamma+[\operatorname{Re}(\gamma)]+1$$ and in the rhs of (\ref{l4.2}), $\alpha\rightarrow-\alpha'$, $\alpha'\rightarrow-\alpha$, $\beta\rightarrow-\beta'$, $\beta'\rightarrow-\beta$, $\gamma\rightarrow-\gamma$ respectively, then Lemma \ref{ll4.1}(b) is obtained.
\end{remark}
\noindent The next result gives the left-hand sided MSM fractional derivative of the $I$-function.
\begin{theorem}\label{t4.1}
Let $\alpha,\alpha',\beta,\beta',\gamma,\rho,a\in\mathbb{C}$ be such that $\mu>0$ and $\operatorname{Re}\left(\rho\right)>\max\{0,$ $\operatorname{Re}(-\alpha+\beta),\operatorname{Re}(-\alpha-\alpha'-\beta'+\gamma)\}$. Then for $x>0$,
\begin{eqnarray}\label{4.1}
&&\left(\mathcal{D}_{0+}^{\alpha,\alpha',\beta,\beta',\gamma}\left(t^{\rho-1}I^{m,n}_{p,q}\Bigg[at^{\mu}\left|
\begin{matrix}
    (a_i,A_i,\alpha_i)_{1,p}\\ 
    (b_j,B_j,\beta_j)_{1,q}
  \end{matrix}
\right.\Bigg]\right)\right)(x)\nonumber\\
&&\ \ \ \ \ \ \ \ \ \ \ =x^{\alpha+\alpha'-\gamma+\rho-1}I^{m,n+3}_{p+3,q+3}\Bigg[ax^{\mu}\left|
\begin{matrix}
    (1-\rho,\mu,1)&(1-\alpha+\beta-\rho,\mu,1)\\ 
    (b_j,B_j,\beta_j)_{1,q}&(1+\beta-\rho,\mu,1)
  \end{matrix}\right.\nonumber\\
&&\ \ \ \ \ \ \ \ \ \ \ \ \ \ \ \ \ \  
\begin{matrix}
   (1-\alpha-\alpha'-\beta'+\gamma-\rho,\mu,1)&(a_i,A_i,\alpha_i)_{1,p}\\ 
   (1-\alpha-\alpha'+\gamma-\rho,\mu,1)&(1-\alpha-\beta'+\gamma-\rho,\mu,1)
  \end{matrix}
\Bigg].
\end{eqnarray}
\end{theorem}
\begin{proof}
On using (\ref{1.2}), the lhs of (\ref{4.1}) equals
\begin{equation}\label{f3.2}
\left(\mathcal{D}_{0+}^{\alpha,\alpha',\beta,\beta',\gamma}\left(t^{\rho-1}\frac{1}{2\pi i}\int_{C}\chi(s)(at^{\mu})^{-s}\,ds,\right)\right)(x),
\end{equation}
where $\chi(s)$ is given by (\ref{n1.1}). Interchanging the order of integration and differentiation and using (\ref{l4.1}), (\ref{f3.2}) becomes
\begin{equation*}
\frac{1}{2\pi i}\int_{C}\chi(s)a^{-s}\left(\mathcal{D}_{0+}^{\alpha,\alpha',\beta,\beta',\gamma}t^{\rho-\mu s-1}\right)(x)\,ds=x^{\alpha+\alpha'-\gamma+\rho-1}\frac{1}{2\pi i}\int_{C}\chi(s)\chi_3(s)(ax^{\mu})^{-s}\,ds,
\end{equation*}
where
\begin{equation*}
\chi_3(s)=\frac{\Gamma\left(\rho-\mu s\right)\Gamma\left(\alpha-\beta+\rho-\mu s\right)\Gamma\left(\alpha+\alpha'+\beta'-\gamma+\rho-\mu s\right)}{\Gamma\left(-\beta+\rho-\mu s\right)\Gamma\left(\alpha+\alpha'-\gamma+\rho-\mu s\right)\Gamma\left(\alpha+\beta'-\gamma+\rho-\mu s\right)}.
\end{equation*}
The result follows from (\ref{1.2}).
\end{proof}
\noindent Following corollaries for Saigo and Erd\'elyi-Kober fractional derivatives follow immediately.
\begin{corollary}\label{c4.1}
Let $\alpha,\beta,\gamma,\rho,a\in\mathbb{C}$ be such that $\mu>0$ and $\operatorname{Re}\left(\rho\right)>\max\{0,\operatorname{Re}(-\alpha-\beta-\gamma)\}$. Then the left-hand sided generalized fractional differentiation $\mathcal{D}_{0+}^{\alpha,\beta,\gamma}$ of the $I$-function is given for $x>0$ by
\begin{eqnarray*}\label{4.3}
&&\left(\mathcal{D}_{0+}^{\alpha,\beta,\gamma}\left(t^{\rho-1}I^{m,n}_{p,q}\Bigg[at^{\mu}\left|
\begin{matrix}
    (a_i,A_i,\alpha_i)_{1,p}\\ 
    (b_j,B_j,\beta_j)_{1,q}
  \end{matrix}
\right.\Bigg]\right)\right)(x)\nonumber\\
&&\ \ =x^{\beta+\rho-1}I^{m,n+2}_{p+2,q+2}\Bigg[ax^{\mu}\left|
\begin{matrix}
    (1-\rho,\mu,1)&(1-\alpha-\beta-\gamma-\rho,\mu,1)&(a_i,A_i,\alpha_i)_{1,p}\\ 
    (b_j,B_j,\beta_j)_{1,q}&(1-\gamma-\rho,\mu,1)&(1-\beta-\rho,\mu,1)
  \end{matrix}\right.
\Bigg].
\end{eqnarray*}
\end{corollary}
\begin{corollary}
Let $\alpha,\gamma,\rho\in\mathbb{C}$ be such that $\mu>0$ and $\operatorname{Re}\left(\rho\right)>\max\{0$, $\operatorname{Re}(-\alpha-\gamma)\}$. Then the left-hand sided Erd\'elyi-Kober fractional differentiation $\mathcal{D}_{\gamma,\alpha}^{+}$ $(=\mathcal{D}_{0+}^{\alpha,0,\gamma})$ of the $I$-function is given for $x>0$ by
\begin{eqnarray*}\label{4.5}
&&\left(\mathcal{D}_{\gamma,\alpha}^{+}\left(t^{\rho-1}I^{m,n}_{p,q}\Bigg[at^{\mu}\left|
\begin{matrix}
    (a_i,A_i,\alpha_i)_{1,p}\\ 
    (b_j,B_j,\beta_j)_{1,q}
  \end{matrix}
\right.\Bigg]\right)\right)(x)\nonumber\\
&&\ \ \ \ \ \ \ \ \ \ \ \ \ \ \ \ \ \ \ \ \ \ \ \ \ \ \ \ \ =x^{\rho-1}I^{m,n+1}_{p+1,q+1}\Bigg[ax^{\mu}\left|
\begin{matrix}
    (1-\alpha-\gamma-\rho,\mu,1)&(a_i,A_i,\alpha_i)_{1,p}\\ 
    (b_j,B_j,\beta_j)_{1,q}&(1-\gamma-\rho,\mu,1)
  \end{matrix}\right.
\Bigg].
\end{eqnarray*}
\end{corollary}
\noindent The next theorem yields the right-hand sided MSM fractional derivative of the $I$-function.
\begin{theorem}\label{t4.2}
Let $\alpha,\alpha',\beta,\beta',\gamma,\rho,a\in\mathbb{C}$ be such that $\mu>0$ and $\operatorname{Re}\left(\rho\right)>\max\{\operatorname{Re}(-\beta'),$ $\operatorname{Re}(\alpha'+\beta-\gamma),\operatorname{Re}(\alpha+\alpha'-\gamma)+[\operatorname{Re}(\gamma)]+1,\}$. Then 
\begin{eqnarray}\label{4.6}
&&\left(\mathcal{D}_{-}^{\alpha,\alpha',\beta,\beta',\gamma}\left(t^{-\rho}I^{m,n}_{p,q}\Bigg[at^{-\mu}\left|
\begin{matrix}
    (a_i,A_i,\alpha_i)_{1,p}\\ 
    (b_j,B_j,\beta_j)_{1,q}
  \end{matrix}
\right.\Bigg]\right)\right)(x)\nonumber\\
&&\ \ \ \ \ \ \ \ \ \ \ \ \ =x^{\alpha+\alpha'-\gamma-\rho}I^{m,n+3}_{p+3,q+3}\Bigg[ax^{-\mu}\left|
\begin{matrix}
    (1-\beta'-\rho,\mu,1)&(1+\alpha+\alpha'-\gamma-\rho,\mu,1)\\ 
    (b_j,B_j,\beta_j)_{1,q}&(1-\rho,\mu,1)
  \end{matrix}\right.\nonumber\\
&&\ \ \ \ \ \ \ \ \ \ \ \ \ \ \ \ \ \ \ \ \ \ \ \ \ 
\begin{matrix}
   (1+\alpha'+\beta-\gamma-\rho,\mu,1)&(a_i,A_i,\alpha_i)_{1,p}\\ 
   (1+\alpha'-\beta'-\rho,\mu,1)&(1+\alpha+\alpha'+\beta-\gamma-\rho,\mu,1)
  \end{matrix}
\Bigg],
\end{eqnarray}
for $x>0$.
\end{theorem}
\begin{proof}
Using (\ref{l4.2}) and the definition of $I$-function (\ref{1.2}), the lhs of (\ref{4.6}) equals
\begin{eqnarray*}
&&\left(\mathcal{D}_{-}^{\alpha,\alpha',\beta,\beta',\gamma}\left(t^{-\rho}\frac{1}{2\pi i}\int_{C}\chi(s)(at^{-\mu})^{-s}\,ds\right)\right)(x)\\
&&\ \ \ \ \ \ \ \ \ \ \ \ \ \ \ \ \ \ \ \ \ \ \ \ \ \ \ \ \ \ \ \ \ \ \ =\frac{1}{2\pi i}\int_{C}\chi(s)a^{-s}\left(\mathcal{D}_{-}^{\alpha,\alpha',\beta,\beta',\gamma}t^{-(\rho-\mu s)}\right)(x)\,ds\\
&&\ \ \ \ \ \ \ \ \ \ \ \ \ \ \ \ \ \ \ \ \ \ \ \ \ \ \ \ \ \ \ \ \ \ \ =x^{\alpha+\alpha'-\gamma-\rho}\frac{1}{2\pi i}\int_{C}\chi(s)\chi_4(s)(ax^{-\mu})^{-s}\,ds,
\end{eqnarray*}
where $\chi(s)$ is given by (\ref{n1.1}) and 
\begin{equation*}
\chi_4(s)=\frac{\Gamma\left(\beta'+\rho-\mu s\right)\Gamma\left(-\alpha-\alpha'+\gamma+\rho-\mu s\right)\Gamma\left(-\alpha'-\beta+\gamma+\rho-\mu s\right)}{\Gamma\left(\rho-\mu s\right)\Gamma\left(-\alpha'+\beta'+\rho-\mu s\right)\Gamma\left(-\alpha-\alpha'-\beta+\gamma+\rho-\mu s\right)}.
\end{equation*}
Thus the theorem is proved using (\ref{1.2}).
\end{proof}
\begin{corollary}\label{c4.4}
Let $\alpha,\beta,\gamma,\rho,a\in\mathbb{C}$ be such that $\mu>0$ and $\operatorname{Re}\left(\rho\right)>\max\{\operatorname{Re}(-\alpha-\gamma),$ $\operatorname{Re}(\beta)+[\operatorname{Re}(\alpha)]+1\}$. Then the right-hand sided generalized fractional differentiation $\mathcal{D}_{-}^{\alpha,\beta,\gamma}$ of $I$-function is given for $x>0$ by
\begin{eqnarray*}\label{4.11}
&&\left(\mathcal{D}_{-}^{\alpha,\beta,\gamma}\left(t^{-\rho}I^{m,n}_{p,q}\Bigg[at^{-\mu}\left|
\begin{matrix}
    (a_i,A_i,\alpha_i)_{1,p}\\ 
    (b_j,B_j,\beta_j)_{1,q}
  \end{matrix}
\right.\Bigg]\right)\right)(x)\nonumber\\
&&=x^{\beta-\rho}I^{m,n+2}_{p+2,q+2}\Bigg[ax^{-\mu}\left|
\begin{matrix}
    (1+\beta-\rho,\mu,1)&(1-\alpha-\gamma-\rho,\mu,1)&(a_i,A_i,\alpha_i)_{1,p}\\ 
    (b_j,B_j,\beta_j)_{1,q}&(1-\rho,\mu,1)&(1+\beta-\gamma-\rho,\mu,1)
  \end{matrix}\right.\Bigg].
\end{eqnarray*}
\end{corollary}
\begin{corollary}\label{c4.6}
Let $\alpha,\gamma,\rho,a\in\mathbb{C}$ be such that $\mu>0$ and $\operatorname{Re}\left(\rho\right)>\max\{[\operatorname{Re}(\alpha)]+1,$ $\operatorname{Re}(-\alpha-\gamma)\}$. Then the right-hand sided Erd\'elyi-Kober fractional differentiation $\mathcal{D}_{\gamma,\alpha}^{-}$ $(=\mathcal{D}_{-}^{\alpha,0,\gamma})$ of $I$-function is given for $x>0$ by
\begin{eqnarray*}\label{4.13}
&&\left(\mathcal{D}_{\gamma,\alpha}^{-}\left(t^{-\rho}I^{m,n}_{p,q}\Bigg[at^{-\mu}\left|
\begin{matrix}
    (a_i,A_i,\alpha_i)_{1,p}\\ 
    (b_j,B_j,\beta_j)_{1,q}
  \end{matrix}
\right.\Bigg]\right)\right)(x)\nonumber\\
&&\ \ \ \ \ \ \ \ \ \ \ \ \ \ \ \ \ \ \ \ \ \ \ \ \ \ \ \ =x^{-\rho}I^{m,n+1}_{p+1,q+1}\Bigg[ax^{-\mu}\left|
\begin{matrix}
    (1-\alpha-\gamma-\rho,\mu,1)&(a_i,A_i,\alpha_i)_{1,p}\\ 
    (b_j,B_j,\beta_j)_{1,q}&(1-\gamma-\rho,\mu,1)
  \end{matrix}\right.\Bigg].
\end{eqnarray*}
\end{corollary}
\begin{remark}
Theorems \ref{t4.1} and \ref{t4.2} respectively follow from Theorems \ref{t3.1} and \ref{t3.2}, in view of Remark \ref{r4.1}.  
\end{remark}

\section{The Caputo-type MSM fractional differentiation of $I$-function}
\setcounter{equation}{0}
Since the Riemann-Liouville derivatives have some drawbacks while applying to real world problems, especially in the context of initial conditions, the Caputo derivative is being extensively used in applications as the initial conditions have physically significance. Rao \textit{et al.} \cite{Rao15}, introduced Caputo-type fractional derivative, which involve Gauss hypergeometric function in the kernel. For $\alpha,\beta,\gamma\in\mathbb{C}$ and $x>0$ with $\operatorname{Re}(\alpha)>0$, the left- and right-hand sided Caputo fractional differential operators associated with Gauss hypergeometric function are defined by
\begin{equation}\label{c5.1}
\left({}^{c}\mathcal{D}_{0+}^{\alpha,\beta,\gamma}f\right)(x)=\left(\mathcal{I}_{0+}^{-\alpha+[\operatorname{Re}(\alpha)]+1,-\beta-[\operatorname{Re}(\alpha)]-1,\alpha+\gamma-[\operatorname{Re}(\alpha)]-1}f^{([\operatorname{Re}(\alpha)]+1)}\right)(x)
\end{equation}
and
\begin{equation}\label{c5.2}
\left({}^{c}\mathcal{D}_{-}^{\alpha,\beta,\gamma}f\right)(x)=(-1)^{[\operatorname{Re}(\alpha)]+1}\left(\mathcal{I}_{-}^{-\alpha+[\operatorname{Re}(\alpha)]+1,-\beta-[\operatorname{Re}(\alpha)]-1,\alpha+\gamma}f^{\left([\operatorname{Re}(\alpha)]+1\right)}\right)(x),
\end{equation}
where $f^{(n)}$ denotes the $n$-th derivative of $f$. The relation between the Caputo-type MSM fractional derivative and the MSM fractional derivative is same as the relation between the Caputo fractional derivative and the Riemann-Liouville fractional derivative.\\
\noindent For $\alpha,\alpha',\beta,\beta',\gamma\in\mathbb{C}$ and $x>0$ with $\operatorname{Re}(\gamma)>0$, we define the left- and right-hand sided Caputo-type MSM fractional differential operators associated with third Appell function as
\begin{equation}\label{5.1}
\left({}^{c}\mathcal{D}_{0+}^{\alpha,\alpha',\beta,\beta',\gamma}f\right)(x)=\left(\mathcal{I}_{0+}^{-\alpha',-\alpha,-\beta'+[\operatorname{Re}(\gamma)]+1,-\beta,-\gamma+[\operatorname{Re}(\gamma)]+1}f^{([\operatorname{Re}(\gamma)]+1)}\right)(x)
\end{equation}
and
\begin{equation}\label{5.2}
\left({}^{c}\mathcal{D}_{-}^{\alpha,\alpha',\beta,\beta',\gamma}f\right)(x)=\left(-1\right)^{[\operatorname{Re}(\gamma)]+1}\left(\mathcal{I}_{-}^{-\alpha',-\alpha,-\beta',-\beta+[\operatorname{Re}(\gamma)]+1,-\gamma+[\operatorname{Re}(\gamma)]+1}f^{({[\operatorname{Re}(\gamma)]+1})}\right)(x),
\end{equation}
respectively. The fractional operators (\ref{5.1}) and (\ref{5.2}) are connected to (\ref{c5.1}) and (\ref{c5.2}) as follows:
\begin{equation}\label{c1.20}
\left({}^{c}\mathcal{D}_{0+}^{0,\alpha',\beta,\beta',\gamma}f\right)(x)=\left({}^{c}\mathcal{D}_{0+}^{\gamma,\alpha'-\gamma,\beta'-\gamma}f\right)(x),\ \ \ \ 
\left({}^{c}\mathcal{D}_{-}^{0,\alpha',\beta,\beta',\gamma}f\right)(x)=\left({}^{c}\mathcal{D}_{-}^{\gamma,\alpha'-\gamma,\beta'-\gamma}f\right)(x).
\end{equation}
In this section, we study the Caputo-type MSM fractional differentiation of the $I$ function. First we state and prove the following lemma.

\begin{lemma}\label{ll5.1}
Let $\alpha,\alpha',\beta,\beta',\gamma,\rho\in\mathbb{C}$ and $m=[\operatorname{Re}(\gamma)]+1$.\\
\noindent(a) If $\operatorname{Re}(\rho)-m>$ $\max\{0,\operatorname{Re}(-\alpha+\beta),\operatorname{Re}(-\alpha-\alpha'-\beta'+\gamma)\}$, then
\begin{eqnarray}\label{l5.1}
&&\left({}^{c}\mathcal{D}_{0+}^{\alpha,\alpha',\beta,\beta',\gamma}t^{\rho-1}\right)(x)\nonumber\\
&&\ \ \ \ \ \ \ =\frac{\Gamma(\rho)\Gamma(\alpha-\beta+\rho-m)\Gamma(\alpha+\alpha'+\beta'-\gamma+\rho-m)}{\Gamma(-\beta+\rho-m)\Gamma(\alpha+\alpha'-\gamma+\rho)\Gamma(\alpha+\beta'-\gamma+\rho-m)}x^{\alpha+\alpha'-\gamma+\rho-1}.\nonumber\\
\end{eqnarray}
\noindent (b) If $\operatorname{Re}(\rho)+m>$ $\max\{\operatorname{Re}(-\beta'),\operatorname{Re}(\alpha'+\beta-\gamma),\operatorname{Re}(\alpha+\alpha'-\gamma)+[\operatorname{Re}(\gamma)]+1\}$, then
\begin{eqnarray}\label{l5.2}
&&\left({}^{c}\mathcal{D}_{-}^{\alpha,\alpha',\beta,\beta',\gamma}t^{-\rho}\right)(x)\nonumber\\
&&\ \ \ \ \ \ \ =\frac{\Gamma\left(\beta'+\rho+m\right)\Gamma\left(-\alpha-\alpha'+\gamma+\rho\right)\Gamma\left(-\alpha'-\beta+\gamma+\rho+m\right)}{\Gamma\left(\rho\right)\Gamma\left(-\alpha'+\beta'+\rho+m\right)\Gamma\left(-\alpha-\alpha'-\beta+\gamma+\rho+m\right)}x^{\alpha+\alpha'-\gamma-\rho}.\nonumber\\
\end{eqnarray} 
\end{lemma}
\begin{proof} 
(a) From (\ref{5.1}), we have
\begin{eqnarray*}
\left({}^{c}\mathcal{D}_{0+}^{\alpha,\alpha',\beta,\beta',\gamma}t^{\rho-1}\right)(x)
&=&\left(\mathcal{I}_{0+}^{-\alpha',-\alpha,-\beta'+m,-\beta,-\gamma+m}\frac{d^m}{\,dt^m}t^{\rho-1}\right)(x)\\
&=&\frac{\Gamma(\rho)}{\Gamma(\rho-m)}\left(\mathcal{I}_{0+}^{-\alpha',-\alpha,-\beta'+m,-\beta,-\gamma+m}t^{\rho-m-1}\right)(x),
\end{eqnarray*}
which on using (\ref{2.4}) gives (\ref{l5.1}).\\
\noindent (b) From (\ref{5.2}), we have
\begin{eqnarray*}\label{l5.7}
\left({}^{c}\mathcal{D}_{-}^{\alpha,\alpha',\beta,\beta',\gamma}t^{-\rho}\right)(x)
&=&\left(-1\right)^m\left(\mathcal{I}_{-}^{-\alpha',-\alpha,-\beta',-\beta+m,-\gamma+m}\frac{d^m}{\,dt^m}t^{-\rho}\right)(x)\\
&=&\frac{\Gamma(\rho+m)}{\Gamma(\rho)}\left(\mathcal{I}_{-}^{-\alpha',-\alpha,-\beta',-\beta+m,-\gamma+m}t^{-\rho-m}\right)(x),
\end{eqnarray*}
and thus (\ref{l5.2}) follows from (\ref{2.5}).
\end{proof}
\noindent We next present the left-hand sided Caputo type MSM fractional derivative of the $I$-function.
\begin{theorem}\label{t5.1}
Let $\alpha,\alpha',\beta,\beta',\gamma,\rho,a\in\mathbb{C}$, $m=[\operatorname{Re}(\gamma)]+1$ be such that $\mu>0$ and $\operatorname{Re}\left(\rho\right)-m>\max\{0,\operatorname{Re}(-\alpha+\beta),\operatorname{Re}(-\alpha-\alpha'-\beta'+\gamma)\}$. Then for $x>0$
\begin{eqnarray*}\label{5.3}
&&\left({}^{c}\mathcal{D}_{0+}^{\alpha,\alpha',\beta,\beta',\gamma}\left(t^{\rho-1}I^{m,n}_{p,q}\Bigg[at^{\mu}\left|
\begin{matrix}
    (a_i,A_i,\alpha_i)_{1,p}\\ 
    (b_j,B_j,\beta_j)_{1,q}
  \end{matrix}
\right.\Bigg]\right)\right)(x)\nonumber\\
&&\ \ \ \ =x^{\alpha+\alpha'-\gamma+\rho-1}I^{m,n+3}_{p+3,q+3}\Bigg[ax^{\mu}\left|
\begin{matrix}
    (1-\rho,\mu,1)&(1-\alpha+\beta-\rho+m,\mu,1)\\ 
    (b_j,B_j,\beta_j)_{1,q}&(1+\beta-\rho+m,\mu,1)
  \end{matrix}\right.\nonumber\\
&&\ \ \ \ \ \ \ \ \ \ \ \ \ \ \ \ \ \ 
\begin{matrix}
   (1-\alpha-\alpha'-\beta'+\gamma-\rho+m,\mu,1)&(a_i,A_i,\alpha_i)_{1,p}\\ 
   (1-\alpha-\alpha'+\gamma-\rho,\mu,1)&(1-\alpha-\beta'+\gamma-\rho+m,\mu,1)
  \end{matrix}
\Bigg].
\end{eqnarray*}
\end{theorem}
\begin{proof}
From (\ref{1.2}) and (\ref{l5.1}), we get
\begin{eqnarray*}
&&\left({}^{c}\mathcal{D}_{0+}^{\alpha,\alpha',\beta,\beta',\gamma}\left(t^{\rho-1}\frac{1}{2\pi i}\int_{C}\chi(s)(at^{\mu})^{-s}\,ds\right)\right)(x)\\
&&\ \ \ \ \ \ \ \ \ \ \ \ \ \ \ \ \ \ \ =\frac{1}{2\pi i}\int_{C}\chi(s)a^{-s}\left({}^{c}\mathcal{D}_{0+}^{\alpha,\alpha',\beta,\beta',\gamma}t^{\rho-\mu s-1}\right)(x)\,ds\\
&&\ \ \ \ \ \ \ \ \ \ \ \ \ \ \ \ \ \ \ =x^{\alpha+\alpha'-\gamma+\rho-1}\frac{1}{2\pi i}\int_{C}\chi(s)\chi_5(s)(ax^{\mu})^{-s}\,ds,
\end{eqnarray*}
where $\chi(s)$ is given by (\ref{n1.1}) and 
\begin{equation*}
\chi_5(s)=\frac{\Gamma\left(\rho-\mu s\right)\Gamma\left(\alpha-\beta+\rho-\mu s-m\right)\Gamma\left(\alpha+\alpha'+\beta'-\gamma+\rho-\mu s-m\right)}{\Gamma\left(-\beta+\rho-\mu s-m\right)\Gamma\left(\alpha+\alpha'-\gamma+\rho-\mu s\right)\Gamma\left(\alpha+\beta'-\gamma+\rho-\mu s-m\right)}.
\end{equation*}
The result now follows from (\ref{1.2}).
\end{proof}
\begin{corollary}\label{c5.1}
Let $\alpha,\beta,\gamma,\rho,a\in\mathbb{C}$, $m=[\operatorname{Re}(\alpha)]+1$ be such that $\mu>0$ and $\operatorname{Re}\left(\rho\right)-m>\max\{0,\operatorname{Re}(-\alpha-\beta-\gamma)\}$. The left-hand sided generalized Caputo fractional differentiation ${}^{c}\mathcal{D}_{0+}^{\alpha,\beta,\gamma}$ of the $I$-function is given for $x>0$ by
\begin{eqnarray*}\label{5.6}
&&\left({}^{c}\mathcal{D}_{0+}^{\alpha,\beta,\gamma}\left(t^{\rho-1}I^{m,n}_{p,q}\Bigg[at^{\mu}\left|
\begin{matrix}
    (a_i,A_i,\alpha_i)_{1,p}\\ 
    (b_j,B_j,\beta_j)_{1,q}
  \end{matrix}
\right.\Bigg]\right)\right)(x)\nonumber\\
&&\ \ \ \ \ \ \ \ \ \ \ \ \ \ \ \ \ \ \ =x^{\beta+\rho-1}I^{m,n+2}_{p+2,q+2}\Bigg[ax^{\mu}\left|
\begin{matrix}
    (1-\rho,\mu,1)\\ 
    (b_j,B_j,\beta_j)_{1,q}
  \end{matrix}\right.\\
&&\ \ \ \ \ \ \ \ \ \ \ \ \ \ \ \ \ \ \ \ \ \ \ \ \ \ \ \ \ \ \ \ \ \ \begin{matrix}
    (1-\alpha-\beta-\gamma-\rho+m,\mu,1)&(a_i,A_i,\alpha_i)_{1,p}\\ 
    (1-\beta-\rho,\mu,1)&(1-\gamma-\rho+m,\mu,1)
  \end{matrix}\Bigg].
\end{eqnarray*}
\end{corollary}
\begin{corollary}\label{c5.3}
Let $\alpha,\gamma,\rho,a\in\mathbb{C}$, $m=[\operatorname{Re}(\alpha)]+1$ be such that $\mu>0$ and $\operatorname{Re}\left(\rho\right)-m>\max\{0,\operatorname{Re}(-\alpha-\gamma)\}$. Then the left-hand sided Caputo-type Erd¶elyi-Kober fractional differentiation ${}^{c}\mathcal{D}_{\gamma,\alpha}^{+}$ $(={}^{c}\mathcal{D}_{0+}^{\alpha,0,\gamma})$ of the $I$-function is given for $x>0$ by
\begin{eqnarray*}\label{5.8}
&&\left({}^{c}\mathcal{D}_{\gamma,\alpha}^{+}\left(t^{\rho-1}I^{m,n}_{p,q}\Bigg[at^{\mu}\left|
\begin{matrix}
    (a_i,A_i,\alpha_i)_{1,p}\\ 
    (b_j,B_j,\beta_j)_{1,q}
  \end{matrix}
\right.\Bigg]\right)\right)(x)\nonumber\\
&&\ \ \ \ \ \ \ \ \ \ \ \ \ \ =x^{\rho-1}I^{m,n+1}_{p+1,q+1}\Bigg[ax^{\mu}\left|
\begin{matrix}
    (1-\alpha-\gamma-\rho+m,\mu,1)&(a_i,A_i,\alpha_i)_{1,p}\\ 
    (b_j,B_j,\beta_j)_{1,q}&(1-\gamma-\rho+m,\mu,1)
  \end{matrix}\right.\Bigg].
\end{eqnarray*}
\end{corollary}
\noindent Finally, we present the right-hand sided Caputo type MSM fractional derivative of the $I$-function.
\begin{theorem}\label{t5.2}
Let $\alpha,\alpha',\beta,\beta',\gamma,\rho,a\in\mathbb{C}$, $m=[\operatorname{Re}(\gamma)]+1$ be such that $\mu>0$ and $\operatorname{Re}\left(\rho\right)+m>\max\{\operatorname{Re}(-\beta'),\operatorname{Re}(\alpha'+\beta-\gamma),\operatorname{Re}(\alpha+\alpha'-\gamma)+m\}$. Then 
\begin{eqnarray*}\label{5.9}
&&\left({}^{c}\mathcal{D}_{-}^{\alpha,\alpha',\beta,\beta',\gamma}\left(t^{-\rho}I^{m,n}_{p,q}\Bigg[at^{-\mu}\left|
\begin{matrix}
    (a_i,A_i,\alpha_i)_{1,p}\\ 
    (b_j,B_j,\beta_j)_{1,q}
  \end{matrix}
\right.\Bigg]\right)\right)(x)\nonumber\\
&&\ \ \ \ \ \ \ \ \ =x^{\alpha+\alpha'-\gamma-\rho}I^{m,n+3}_{p+3,q+3}\Bigg[ax^{-\mu}\left|
\begin{matrix}
    (1-\beta'-\rho-m,\mu,1)&(1+\alpha+\alpha'-\gamma-\rho,\mu,1)\\ 
    (b_j,B_j,\beta_j)_{1,q}&(1-\rho,\mu,1)
  \end{matrix}\right.\nonumber\\
&&\ \ \ \ \ \ \ \ \ \ \ \ \ \ \ \ \ \ \ 
\begin{matrix}
   (1+\alpha'+\beta-\gamma-\rho-m,\mu,1)&(a_i,A_i,\alpha_i)_{1,p}\\ 
   (1+\alpha'-\beta'-\rho-m,\mu,1)&(1+\alpha+\alpha'+\beta-\gamma-\rho-m,\mu,1)
  \end{matrix}
\Bigg],
\end{eqnarray*}
for $x>0$.
\end{theorem}
\begin{proof}
Using (\ref{l5.2}) and the definition of $I$-function (\ref{1.2}), we have
\begin{eqnarray*}
&&\left({}^{c}\mathcal{D}_{-}^{\alpha,\alpha',\beta,\beta',\gamma}\left(t^{-\rho}\frac{1}{2\pi i}\int_{C}\chi(s)(at^{-\mu})^{-s}\,ds\right)\right)(x)\\
&&\ \ \ \ \ \ \ \ \ \ \ \ \ \ \ \ \ \ \ =\frac{1}{2\pi i}\int_{C}\chi(s)a^{-s}\left({}^{c}\mathcal{D}_{-}^{\alpha,\alpha',\beta,\beta',\gamma}t^{-(\rho-\mu s)}\right)(x)\,ds\\
&&\ \ \ \ \ \ \ \ \ \ \ \ \ \ \ \ \ \ \ =x^{\alpha+\alpha'-\gamma-\rho}\frac{1}{2\pi i}\int_{C}\chi(s)\chi_6(s)(ax^{-\mu})^{-s}\,ds,
\end{eqnarray*}
where $\chi_6(s)$ is given by 
\begin{equation*}
\chi_6(s)=\frac{\Gamma\left(\beta'+\rho-\mu s+m\right)\Gamma\left(-\alpha-\alpha'+\gamma+\rho-\mu s\right)\Gamma\left(-\alpha'-\beta+\gamma+\rho-\mu s+m\right)}{\Gamma\left(\rho-\mu s\right)\Gamma\left(-\alpha'+\beta'+\rho-\mu s+m\right)\Gamma\left(-\alpha-\alpha'-\beta+\gamma+\rho-\mu s+m\right)}.
\end{equation*}
The theorem is now proved using (\ref{1.2}).
\end{proof}
\begin{corollary}\label{c5.4}
Let $\alpha,\beta,\gamma,\rho,a\in\mathbb{C}$, $m=[\operatorname{Re}(\alpha)]+1$ be such that $\mu>0$ and  $\operatorname{Re}\left(\rho\right)+m>\max\{\operatorname{Re}(\beta)+m,\operatorname{Re}(-\alpha-\gamma)\}$. Then the right-hand sided generalized Caputo fractional differentiation ${}^{c}\mathcal{D}_{-}^{\alpha,\beta,\gamma}$ of the $I$-function is given for $x>0$ by
\begin{eqnarray*}\label{5.10}
&&\left({}^{c}\mathcal{D}_{-}^{\alpha,\beta,\gamma}\left(t^{\rho-1}I^{m,n}_{p,q}\Bigg[at^{\mu}\left|
\begin{matrix}
    (a_i,A_i,\alpha_i)_{1,p}\\ 
    (b_j,B_j,\beta_j)_{1,q}
  \end{matrix}
\right.\Bigg]\right)\right)(x)\\
&&\ \ \ \ \ \ \ \ \ \ \ \ \ \ \ \ \ \ \ =x^{\beta-\rho}I^{m,n+2}_{p+2,q+2}\Bigg[ax^{-\mu}\left|
\begin{matrix}
    (1+\beta-\rho,\mu,1)\\ 
    (b_j,B_j,\beta_j)_{1,q}
  \end{matrix}\right.\\
&&\ \ \ \ \ \ \ \ \ \ \ \ \ \ \ \ \ \ \ \ \ \ \ \ \ \ \ \ \ \begin{matrix}
    (1-\alpha-\gamma-\rho-m,\mu,1)&(a_i,A_i,\alpha_i)_{1,p}\\ 
    (1-\rho,\mu,1)&(1+\beta-\gamma-\rho-m,\mu,1)
  \end{matrix}\Bigg].
\end{eqnarray*}
\end{corollary}
\begin{corollary}\label{c5.6}
Let $\alpha,\gamma,\rho,a\in\mathbb{C}$, $m=[\operatorname{Re}(\alpha)]+1$ be such that $\mu>0$ and $\operatorname{Re}\left(\rho\right)+m>\max\{m,\operatorname{Re}(-\alpha-\gamma)\}$. Then the right-hand sided Caputo-type Erd¶elyi-Kober fractional differentiation ${}^{c}\mathcal{D}_{\gamma,\alpha}^{-}$ $(={}^{c}\mathcal{D}_{-}^{\alpha,0,\gamma})$ of the $I$-function is given for $x>0$ by
\begin{eqnarray*}\label{5.12}
&&\left({}^{c}\mathcal{D}_{\gamma,\alpha}^{-}\left(t^{\rho-1}I^{m,n}_{p,q}\Bigg[at^{\mu}\left|
\begin{matrix}
    (a_i,A_i,\alpha_i)_{1,p}\\ 
    (b_j,B_j,\beta_j)_{1,q}
  \end{matrix}
\right.\Bigg]\right)\right)(x)\nonumber\\
&&\ \ \ \ \ \ \ \ \ \ \ \ \ \ \ =x^{-\rho}I^{m,n+1}_{p+1,q+1}\Bigg[ax^{-\mu}\left|
\begin{matrix}
    (1-\alpha-\gamma-\rho-m,\mu,1)&(a_i,A_i,\alpha_i)_{1,p}\\ 
    (b_j,B_j,\beta_j)_{1,q}&(1-\gamma-\rho-m,\mu,1)
  \end{matrix}\right.\Bigg].
\end{eqnarray*}
\end{corollary}
\section{Conclusion}
The $I$-function is one of the most generalized function available in literature, which generalizes $\overline{H}$-function, $H$-function, Meijer $G$-function, generalized Wright function, hypergeometric function, generalized Mittag Leffler function and many other functions. On the other hand, the MSM fractional operators generalize, among others, Saigo, Riemann-Liouville, Weyl and Erd\'elyi-Kober fractional operators. In view of this fact, several recent results obtained by Srivastava \textit{et.al.} \cite{Srivastava94}, Purohit \textit{et.al.} \cite{Purohit21}, Kilbas and Sebastian in series of papers \cite{Kilbas113},\cite{Kilbas159},\cite{Kilbas869},\cite{Kilbas323} and many more become particular cases of our results.
\printbibliography
\end{document}